\documentclass[11pt]{amsart}

\usepackage{hyperref}
\usepackage[usenames]{color}
\usepackage{amsmath,amsthm,amsfonts,amssymb}

\newtheorem{theorem}{Theorem}[section]
\newtheorem{lemma}[theorem]{Lemma}

\newtheorem{definition}[theorem]{Definition}

\theoremstyle{definition}
\newtheorem{example}[theorem]{Example}

\theoremstyle{remark}

\numberwithin{equation}{section}

\begin{document}

\title{Adjunction of roots to unitriangular groups over prime finite fields}

\author{Anton Menshov}
\address{Institute of Mathematics and Information Technologies\\Omsk State Dostoevskii University}
\curraddr{}
\email{menshov.a.v@gmail.com}
\thanks{}

\author{Vitali\u\i\ Roman'kov}
\address{Institute of Mathematics and Information Technologies\\Omsk State Dostoevskii University}
\curraddr{}
\email{romankov48@mail.ru}
\thanks{}

\keywords{equations over groups, nilpotent groups, $p$-groups, unitriangular groups, wreath product}

\date{}

\begin{abstract}
In this paper we study embeddings of unitriangular groups UT$_n(\mathbb{F}_p)$ arising under adjunction of roots.
We construct embeddings of UT$_n(\mathbb{F}_p)$ in UT$_m(\mathbb{F}_p)$, for $n \geq 2$, $m=(n-1)p^s + 1$, $s \in \mathbb{Z}^+$, such that any element of UT$_n(\mathbb{F}_p)$ has a $p^s$-th root in UT$_m(\mathbb{F}_p)$.
Also we construct an embedding of the wreath product UT$_n(\mathbb{F}_p) \wr C_{p^s}$ in UT$_m(\mathbb{F}_p)$, where $C_{p^s}$ is the cyclic group of order $p^s$.
\end{abstract}
\maketitle

\section{Introduction}
\label{sec:intro}

Equations over groups is old and well-established area of the group theory.
B.~H.~Neumann  started its systematic investigation in ~\cite{neumann}.
We refer to the survey~\cite{romankov} for developments in this area.
Also a brief historical note could be found in~\cite{magnus_karrass_solitar}.

An equation with the variable $x$ over a group $G$ is an expression of the form
\begin{equation}\label{eq:one_var_eq}
w(x)=1,
\end{equation}
where $w(x) \in G * \langle x \rangle$ is a group word formed with $x$ and elements of $G$.

If $H$ is a bigger group, i.e., a group containing $G$ as a fixed subgroup, then an equation over $G$ could be also considered as an equation over $H.$
Equation~(\ref{eq:one_var_eq}) is {\it solvable in $G$} if there is an element $g \in G$ such that $w(g)=1$.
Equation~(\ref{eq:one_var_eq}) is {\it solvable over $G$} if there is an overgroup $H \geq G$ where this equation has a solution.
In the latter case we may assume that $H$ is generated by a solution of~(\ref{eq:one_var_eq}) and elements of $G$.
In other words, $H$ is obtained by {\it adjoining} a solution of~(\ref{eq:one_var_eq}) to $G$.
We may also call $H$ an {\it extension} of $G$.

In \cite{neumann} B.~H.~Neumann studied conditions on $G$ and $w(x)$ under which equation~(\ref{eq:one_var_eq}) is solvable over $G$ and obtained solution of this problem in general case~\cite[Theorem 2.3]{neumann}.
Also he proved the following:

\begin{theorem}[B.~H.~Neumann,~\cite{neumann}]
\label{th:neumann_01}
The equation
\begin{equation}
\label{eq:power}
x^m = g
\end{equation}
is solvable over an arbitrary group $G$ for any $g \in G$ and $m \in \mathbb{Z}^+$.
\end{theorem}

Solution of~(\ref{eq:power}) is called {\it $m$-th root} of $g$.
To prove Theorem~\ref{th:neumann_01} B.~H.~Neumann used amalgamated free products.
An alternative way using wreath products was suggested by G.~Baumslag in~\cite{baumslag_01}.

From Theorem~\ref{th:neumann_01} we obtain the following result:

\begin{theorem}[B.~H.~Neumann,~\cite{neumann}]
\label{th:neumann_02}
Every group $G$ is isomorphic to a subgroup of a group $D$ in which every element has an $n$-th root for every $n \in \mathbb{Z}^+$.
\end{theorem}

Such a group $D$ is called a {\it divisible} group.

According to~\cite{wiegold}, in 1960 B.~H.~Neumann posed the following problem: given a nilpotent group $B$, an element $h$ of $B$ and a positive integer $n$, is it always possible to embed $B$ in a nilpotent group containing an $n$-th root for $h$?
In other words, whether {\it nilpotent} adjunction of $n$-th root to $h$ is possible?
This problem was studied in~\cite{wiegold,allenby}.
In particular J.~Wiegold in~\cite{wiegold} proved that it is always possible to nilpotently adjoint $n$-th roots to elements of finite order.
We remark that the answer to the problem posed by B.~H.~Neumann is, in general, negative.
For suppose that $p$ is any given prime.
Then there exists a nilpotent group $G$ (infinitely generated) with an element $u$ (of infinite order) such that any overgroup $H$ of $G$ in which $u$ has a $p$-th root is not nilpotent~\cite{neumann_wiegold}.

We will mention some results of G. Baumslag~\cite{baumslag_04} related to the problem above.

\begin{theorem}[G.~Baumslag,~\cite{baumslag_04}]
\label{th:baumslag_01}
If $p$ is any prime and $G$ any given finitely generated nilpotent group, then $G$ can be embedded in a nilpotent group $H$ so that any element $u$ in $G$ now has a $p$-th root in $H$.
\end{theorem}

\begin{theorem}[G.~Baumslag,~\cite{baumslag_04}]
\label{th:baumslag_02}
Any finitely generated nilpotent group can be embedded in a locally nilpotent group which is divisible.
\end{theorem}

We also remark that in~\cite{malcev} A.~Mal'cev proved that any torsion free nilpotent group can be embedded in a divisible nilpotent group of the same class.

Let $\mathbb{F}_p$ be the prime finite field of order $p$ and UT$_n(\mathbb{F}_p)$ ($n \geq 2$) be the group of $n \times n$ upper unitriangular matrices over $\mathbb{F}_p$.
In this paper we will be interested in adjunction of $p^s$-th roots, $s \in \mathbb{Z}^+$, to the group UT$_n(\mathbb{F}_p)$, such that an overgroup is of the form UT$_m(\mathbb{F}_p)$ for some $m \geq n$.
Observe that we only need to perform adjunction of $p^s$-th roots to a finite $p$-group $G$, since $G$ contains $n$-th roots for any $n$ such that $\gcd(n,p)=1$.

In the group UT$_n(\mathbb{F}_p)$ by $t_{i,j}$ ($1 \leq i < j \leq n$) we denote a {\it transvection} $e + e_{i,j}$, where $e$ is the identity matrix, $e_{i,j}$ is the matrix having $1$ in the $ij$-component and $0$ otherwise.
Also for any $\gamma \in \mathbb{F}_p$ we denote $t_{i,j}(\gamma) = e + \gamma e_{i,j}$.

In Lemma~\ref{lm:root} we perform adjunction of roots to transvections in UT$_n(\mathbb{F}_p)$.
In Theorem~\ref{th:roots_in_UT_n} we perform {\it simultaneous} adjunction of $q$-th roots ($q=p^s$, $s \in \mathbb{Z}^+$) to UT$_n(\mathbb{F}_p)$, i.e., we construct the embeddings of UT$_n(\mathbb{F}_p)$ in UT$_m(\mathbb{F}_p)$, where $m=(n-1)q + 1$, such that any element of UT$_n(\mathbb{F}_p)$ has a $q$-th root in UT$_m(\mathbb{F}_p)$.

It is well known that any finite $p$-group $G$ is isomorphic to a subgroup of UT$_{|G|}(\mathbb{F}_p)$.
Observe that the wreath product UT$_n(\mathbb{F}_p) \wr C_q$ of the group UT$_n(\mathbb{F}_p)$ with the cyclic group of order $q$ contains a $q$-th root for any element of UT$_n(\mathbb{F}_p)$ and is also a finite $p$-group.
So it embeds in UT$_m(\mathbb{F}_p)$ for some $m > n$.
In Theorem~\ref{th:main} we construct the embedding of UT$_n(\mathbb{F}_p) \wr C_q$ in UT$_m(\mathbb{F}_p)$, where $m = (n-1)q+1$.
This value of $m$ is the minimal possible, since, according to Lemma~\ref{lm:wp_class}, the nilpotency class of UT$_n(\mathbb{F}_p) \wr C_q$ is equal to $(n-1)q$.
In Lemma~\ref{lm:equiv} we show that theorems~\ref{th:roots_in_UT_n} and~\ref{th:main} lead to the same result.

\section{Preliminaries}
\label{sec:preliminaries}

We will outline the definition of the wreath product $G \wr C$.
Let $G^C$ be the group of all mappings from $C$ to $G$ with multiplication defined by $(f \cdot f')(t) = f(t)f'(t)$, for all $t \in C$.
The group $G^C$ is called the base group of the wreath product.
The group $G \wr C$ is the set of pairs
$
\{ sf \mid s \in C, f \in G^C \}
$
with multiplication
\[
sf \cdot s'f' = ss' f^{s'} f',
\]
where $f^{s'}(t) = f(t s'^{-1})$, for all $t \in C$.

If $C=C_n(c)$ is the cyclic group of order $n$, generated by $c$, then an element $f$ of the base group is essentially a tuple of length $n$
\[
(f(c^0),f(c^1),\dots,f(c^{n-1})).
\]
An element $f^c$ is a tuple
\[
(f(c^{n-1}),f(c^0),\dots,f(c^{n-2})),
\]
i.e., the generator $c$ of the cyclic group acts on elements of the base group as the right cyclic shift.

Using wreath products for adjunction of roots, the group $G$ is usually identified with the diagonal subgroup of the base group.

In~\cite{baumslag_01} G.~Baumslag proved that the wreath product $A \wr B$ of two nontrivial nilpotent groups is nilpotent if and only if both $A$ and $B$ are $p$-groups with $A$ of finite exponent and $B$ finite.
Since that there were many attempts to find the class of a nilpotent wreath product.
It was finally obtained by D.~Shield (see~\cite{meldrum, shield_1, shield_2} for details).

\begin{definition}
\label{def:jennings_series}
Let $G$ be a group and $p$ be a prime.
The $K_p$-series of $G$ is defined by
\[
K_{i,p}(G) = \prod_{n p^j \geq i} \gamma_n (G)^{p^j},
\]
where $i \geq 1$ and $\gamma_n (G)$ is the $n$-th term of the lower central series of $G$.
In particular $K_{1,p}(G)=G$.
\end{definition}

The $K_p$-series of a finite $p$-group will eventually reach the trivial group.
Hence the following definition makes sense.

\begin{definition}
\label{def:jennings_series_numbers}
Let $B$ be a finite $p$-group, for a prime $p$.
Let $d$ be the maximal integer such that $K_{d,p}(B) \neq \{1\}$.
For each $v$, $v=1,\dots,d$, define $e(v)$ by
\[
p^{e(v)} = | K_{v,p}(B) / K_{v+1,p}(B) |
\]
and define $a$ and $b$ by
\begin{align*}
a &= 1 + (p-1) \sum_{v=1}^{d} v e(v), \\
b &= (p-1)d.
\end{align*}
\end{definition}

We can now state Shield's result.

\begin{theorem}[D.~Shield,~\cite{shield_2}]
\label{th:shield}
Let $p$ be a prime, $A$ a $p$-group, nilpotent of class $c$, and of finite exponent, and let $B$ be a finite $p$-group, with $a$ and $b$ defined as in Definition~\ref{def:jennings_series_numbers}.
Define $s(w)$ by $p^{s(w)}$ is the exponent of $\gamma_w(A)$, for $w=1,\dots,c$.
Then
\[
\mathrm{cl}(A \wr B) = \max_{1 \leq w \leq c} \{ aw + (s(w)-1)b \},
\]
where $\mathrm{cl}(G)$ is the nilpotency class of $G$.
\end{theorem}

We will apply the theorem above to prove the following lemma.

\begin{lemma}\label{lm:wp_class}
Let $p$ be a prime, $q=p^s$, $s \in \mathbb{Z}^+$, and $n \geq 2$, then
\[
\mathrm{cl}(\mathrm{UT}_n(\mathbb{F}_p) \wr C_q) = q(n-1).
\]
\end{lemma}

\begin{proof}
First we will compute for the group $G=C_q$ its $K_p$-series (by Definition~\ref{def:jennings_series}).
It is easy to check that $K_p$-series of $G$ has the form
\begin{align*}
K_{1,p} =
G &\geq \overbrace{G^p \geq \dots \geq G^p}^{p-1}
  \geq \overbrace{G^{p^2} \geq \dots \geq G^{p^2}}^{p(p-1)} \geq \dots \\
  &\geq \overbrace{G^{p^{s-1}} \geq \dots \geq G^{p^{s-1}}}^{p^{s-2}(p-1)}
  \geq K_{p^{s-1}+1,p} = \{1\}.
\end{align*}
Then, using notation of Definition~\ref{def:jennings_series_numbers}, the sequence $e(v)$ ($v=1,\dots,d$) has the form
\[
\overbrace{1,0,\dots,0}^{p-1},
\overbrace{1,0,\dots,0}^{p(p-1)},
\dots,
\overbrace{1,0,\dots,0}^{p^{s-2}(p-1)},
1
\]
and $d=p^{s-1}$, $b=(p-1)p^{s-1}$,
\[
a = 1 + (p-1) \sum_{v=1}^{d} v e(v) = 1 + (p-1) (1 + p + p^2 + \dots + p^{s-1}) = p^s.
\]

\noindent
By Theorem~\ref{th:shield} the nilpotency class of the wreath product UT$_n(\mathbb{F}_p) \wr C_q$ is determined by
\[
\mathrm{cl}(\mathrm{UT}_n(\mathbb{F}_p) \wr C_q) =
    \max_{1 \leq w \leq n-1} \{p^s w + (s(w)-1)(p-1)p^{s-1}\},
\]
where $p^{s(w)}$ is the exponent of $\gamma_w(\mathrm{UT}_n(\mathbb{F}_p))$.
Observe that $\gamma_{n-1}(\mathrm{UT}_n(\mathbb{F}_p))$ is the cyclic group of order $p$, hence $s(n-1)=1$.
It is easy to see that maximum is attained when $w=n-1$, so
\[
\mathrm{cl}(\mathrm{UT}_n(\mathbb{F}_p) \wr C_q) =
    p^s(n-1) = q(n-1).
\]

\end{proof}

\section{Adjunction of roots}
\label{sec:adjunction_of_roots}

We will introduce the following notations for the subgroups of UT$_n(\mathbb{F}_p)$:
\begin{equation}\label{eq:subgroups_of_UT}
\begin{aligned}
\mathrm{FR}_n &= \{ (a_{ij}) \in \mathrm{UT}_n(\mathbb{F}_p) \mid a_{ij}=0, j > i > 1 \}, \\
\mathrm{LC}_n &= \{ (a_{ij}) \in \mathrm{UT}_n(\mathbb{F}_p) \mid a_{ij}=0, n > j > i \}, \\
\mathrm{A}_n  &= \{ (a_{ij}) \in \mathrm{UT}_n(\mathbb{F}_p) \mid a_{1j}=0, j >1 \}, \\
\mathrm{B}_n  &= \{ (a_{ij}) \in \mathrm{UT}_n(\mathbb{F}_p) \mid a_{in}=0, i < n \}.
\end{aligned}
\end{equation}

\noindent
Observe that subgroups FR$_n$ and LC$_n$ are normal in UT$_n(\mathbb{F}_p)$ and isomorphic to $C_p^{n-1}$, where $C_p$ is the cyclic group of order $p$.
Subgroups A$_n$ and B$_n$ are naturally isomorphic to UT$_{n-1}(\mathbb{F}_p)$.
Furthermore
\[
\mathrm{UT}_n(\mathbb{F}_p)
    = \mathrm{FR}_n \leftthreetimes \mathrm{A}_n
    = \mathrm{LC}_n \leftthreetimes \mathrm{B}_n.
\]

\noindent
Let $t_{i,j}$ denote a transvection in the group UT$_n(\mathbb{F}_p)$.
It is known that $t_{i,j}$ satisfy the following relations
\begin{equation}\label{eq:UT_relations}
\begin{aligned}
    ~[t_{i,j},t_{j,k}] &= t_{i,k}, \\
    [t_{i,j},t_{k,l}] &= 1, \qquad j \neq k, \; i \neq l, \\
    t_{i,j}^p &= 1,
\end{aligned}
\end{equation}
and any other relation between elements of UT$_n(\mathbb{F}_p)$ is a consequence of relations~(\ref{eq:UT_relations}).

Let $m>n$ and $1=k_1 < k_2 \dots < k_n=m$ be a sequence of integers.
By $t_{i,j}'$ we denote a transvection in UT$_m(\mathbb{F}_p)$.
It is clear that the mapping
\begin{equation}
\label{eq:simple_UT_embedding}
\phi: t_{i,i+1} \mapsto t_{k_i,k_{i+1}}', \quad i=1,\dots,n-1,
\end{equation}
is an embedding of UT$_n(\mathbb{F}_p)$ in UT$_m(\mathbb{F}_p)$.

The following lemma utilizes the embedding above to perform adjunction of roots to transvections in UT$_n(\mathbb{F}_p)$.
Further we will use rational numbers to index rows and columns of matrices.

\begin{lemma}
\label{lm:root}
The equation

\begin{equation}
\label{eq:transvection_root}
x^{p^s r} = t_{i,j}(\gamma),
\end{equation}

\noindent
over UT$_n(\mathbb{F}_p)$ ($n \geq 2$),
where $\gcd(p,r)=1$, $s \in \mathbb{Z}^+$,
is solvable in an overgroup isomorphic to UT$_m(\mathbb{F}_p)$, where $m = n + p^s - 1.$
\end{lemma}

\begin{proof}
For brevity denote $q=p^s$.
It is well known that UT$_n(\mathbb{F}_p)$ is generated by $t_{i,i+1}$, for $i=1,\dots,n-1$.
Insert a sequence of $q-1$ numbers $\alpha_l \in \mathbb{Q} \setminus \mathbb{N}$ between $i$ and $j$ such that
\[
i < \alpha_1 < \dots < \alpha_{q-1} < j.
\]
Positions of $\alpha_l$ relative to indices $i+1,\dots,j-1$ don't have much effect.
For simplicity we assume that  $1 < \alpha_l < i+1$, for $l=1,\dots,q-1$.

Let UT$_n(\mathbb{F}_p)$ be embedded in UT$_m(\mathbb{F}_p)$, generated by transvections
\[
t_{1,2}', \dots, t_{i-1,i}',
t_{i,\alpha_1}', t_{\alpha_1,\alpha_2}', \dots, t_{\alpha_{q-1},i+1}',
t_{i+1,i+2}', \dots, t_{n-1,n}',
\]
by the mapping
\begin{equation}
\label{eq:transvection_embedding}
\phi: t_{i,i+1} \mapsto t_{i,i+1}', \quad i=1,\dots,n-1.
\end{equation}
In UT$_m(\mathbb{F}_p)$ take the element
\[
x = e + \gamma_1 e_{i,\alpha_1} + \dots + \gamma_q e_{\alpha_{q-1},j},
\]
where $\gamma_1 \gamma_2 \dots \gamma_q = r^{-1} \gamma $ ($r^{-1}$ is the inverse for $r$ modulo $p$), then
\[
x^{qr} = (e + r^{-1} \gamma e_{i,j})^r = 
    e + \gamma e_{i,j} = t_{i,j}(\gamma ).
\]
Hence $x$ is a solution of~(\ref{eq:transvection_root}).
\end{proof}

The method above, however, doesn't allow to adjoint a root to any element of UT$_n(\mathbb{F}_p)$.
Now for $q=p^s$, $s \in \mathbb{Z}^+$, we will describe the embeddings of UT$_n(\mathbb{F}_p)$ in UT$_m(\mathbb{F}_p)$, where $m=(n-1)q+1$, that naturally arise in Theorem~\ref{th:roots_in_UT_n}.

Let $\alpha_{i,j} \in \mathbb{Q}$ be such that
\[
i < \alpha_{i,1} < \dots < \alpha_{i,q-1} < i+1, \qquad i=1,\dots,n-1,
\]
and let UT$_m(\mathbb{F}_p)$ be generated by
\[
t_{i, \alpha_{i,1}}', t_{\alpha_{i,1}, \alpha_{i,2}}', \dots, t_{\alpha_{i,q-1}, i+1}', \qquad i=1,\dots,n-1.
\]
Consider the embedding $\phi:$ UT$_n(\mathbb{F}_p) \to$ UT$_m(\mathbb{F}_p)$ defined by
\begin{equation}
\label{eq:UT_embedding_FR}
\phi: \begin{array}{rcl}
    t_{1,2} & \mapsto & t_{1,2}', \\
    t_{2,3} & \mapsto & t_{2,3}' t_{\alpha_{1,1}, \alpha_{2,1}}' t_{\alpha_{1,2}, \alpha_{2,2}}' \dots t_{\alpha_{1,q-1}, \alpha_{2,q-1}}', \\
    t_{3,4} & \mapsto & t_{3,4}' t_{\alpha_{2,1}, \alpha_{3,1}}' t_{\alpha_{2,2}, \alpha_{3,2}}' \dots t_{\alpha_{2,q-1}, \alpha_{3,q-1}}', \\
    & \dots & \\
    t_{n-1,n} & \mapsto & t_{n-1,n}' t_{\alpha_{n-2,1}, \alpha_{n-1,1}}' t_{\alpha_{n-2,2}, \alpha_{n-1,2}}' \dots t_{\alpha_{n-2,q-1}, \alpha_{n-1,q-1}}'.
\end{array}
\end{equation}
We will prove that~(\ref{eq:UT_embedding_FR}) is really an embedding.

Define the following subgroups in UT$_m(\mathbb{F}_p)$
\begin{align*}
H_i &= \left\langle t_{\alpha_{1,i}, \alpha_{2,i}}', t_{\alpha_{2,i}, \alpha_{3,i}}', \dots, t_{\alpha_{n-2,i}, \alpha_{n-1,i}}' \right\rangle, \quad i=1,\dots,q-1, \\
H_q &= \left\langle t_{2,3}', t_{3,4}', \dots, t_{n-1,n}' \right\rangle.
\end{align*}
Observe that $H_i\simeq$ UT$_{n-1}(\mathbb{F}_p)$, for $i=1,\dots,q$, and for $k \neq l$ subgroups $H_k$ and $H_l$ are commuting element-wise.
Denote $H = H_1 \times \dots \times H_q$ and $D(H)$ is the diagonal subgroup of $H$, then $D(H)\simeq$ UT$_{n-1}(\mathbb{F}_p)$.
Consider a subgroup $W$ of FR$_m$, consisting of matrices having nonzero elements only in positions $kq+1$, for $k=0,\dots,n-1$. 
It is easy to see that $W\simeq$ FR$_n$.
Take a semidirect product $P = W \leftthreetimes D(H)$, where an element $(h_1,\dots,h_q) \in D(H)$ acts on $w \in W$ as a conjugation by $h_q$.
Then $P\simeq$ FR$_n \; \leftthreetimes$ UT$_{n-1}(\mathbb{F}_p) \simeq$ UT$_n(\mathbb{F}_p)$.
The basis of $P$ consist of images $\phi(t_{i,i+1})$, $i=1,\dots,n-1$, specified in~(\ref{eq:UT_embedding_FR}).
Hence~(\ref{eq:UT_embedding_FR}) is really an embedding.

\begin{example}
Let $n=p=q=3$, then the image of
\[
a = \left( \begin{smallmatrix}
1 & a_{12} & a_{13} \\
0 & 1      & a_{23} \\
0 & 0      & 1
\end{smallmatrix} \right) \in \mathrm{UT}_3(\mathbb{F}_3)
\]
under embedding~(\ref{eq:UT_embedding_FR}) is equal to
\[
\phi(a) = \left( \begin{smallmatrix}
1 & 0 & 0 & a_{12} & 0 & 0 & a_{13} \\
0 & 1 & 0 & 0 & a_{23} & 0 & 0 \\
0 & 0 & 1 & 0 & 0 & a_{23} & 0 \\
0 & 0 & 0 & 1 & 0 & 0 & a_{23} \\
0 & 0 & 0 & 0 & 1 & 0 & 0 \\
0 & 0 & 0 & 0 & 0 & 1 & 0 \\
0 & 0 & 0 & 0 & 0 & 0 & 1
\end{smallmatrix} \right) \in \mathrm{UT}_7(\mathbb{F}_3)
\]
\end{example}

\medskip

Similarly one can define the embedding $\psi:$ UT$_n(\mathbb{F}_p) \to$ UT$_m(\mathbb{F}_p)$ by
\begin{equation}
\label{eq:UT_embedding_LC}
\psi: \begin{array}{rcl}
    t_{1,2} & \mapsto & t_{1,2}' t_{\alpha_{1,1}, \alpha_{2,1}}' t_{\alpha_{1,2}, \alpha_{2,2}}' \dots t_{\alpha_{1,q-1}, \alpha_{2,q-1}}', \\
    t_{2,3} & \mapsto & t_{2,3}' t_{\alpha_{2,1}, \alpha_{3,1}}' t_{\alpha_{2,2}, \alpha_{3,2}}' \dots t_{\alpha_{2,q-1}, \alpha_{3,q-1}}', \\
        & \dots & \\ 
        t_{n-2,n-1} & \mapsto & t_{n-2,n-1}' t_{\alpha_{n-2,1}, \alpha_{n-1,1}}' t_{\alpha_{n-2,2}, \alpha_{n-1,2}}' \dots t_{\alpha_{n-2,q-1}, \alpha_{n-1,q-1}}', \\
        t_{n-1,n} & \mapsto & t_{n-1,n}'.
\end{array}
\end{equation}
In a similar way one can prove that~(\ref{eq:UT_embedding_LC}) is really an embedding.

\begin{example}
Let $n=p=q=3$, then the image of
\[
a = \left( \begin{smallmatrix}
1 & a_{12} & a_{13} \\
0 & 1      & a_{23} \\
0 & 0      & 1
\end{smallmatrix} \right) \in \mathrm{UT}_3(\mathbb{F}_3)
\]
under embedding~(\ref{eq:UT_embedding_LC}) is equal to
\[
\psi(a) = \left( \begin{smallmatrix}
1 & 0 & 0 & a_{12} & 0 & 0 & a_{13} \\
0 & 1 & 0 & 0 & a_{12} & 0 & 0 \\
0 & 0 & 1 & 0 & 0 & a_{12} & 0 \\
0 & 0 & 0 & 1 & 0 & 0 & a_{23} \\
0 & 0 & 0 & 0 & 1 & 0 & 0 \\
0 & 0 & 0 & 0 & 0 & 1 & 0 \\
0 & 0 & 0 & 0 & 0 & 0 & 1
\end{smallmatrix} \right) \in \mathrm{UT}_7(\mathbb{F}_3)
\]
\end{example}

\medskip

Using embeddings~(\ref{eq:UT_embedding_FR}) and~(\ref{eq:UT_embedding_LC}) one can adjoint roots to any elements of UT$_n(\mathbb{F}_p)$.

\begin{theorem}
\label{th:roots_in_UT_n}
The equation

\begin{equation}
\label{eq:power_p_eq_over_UT_n}
x^q = a,
\end{equation}

\noindent
over UT$_n(\mathbb{F}_p)$ ($n \geq 2$),
where $q=p^s$, $s \in \mathbb{Z}^+$, is solvable in an overgroup isomorphic to UT$_m(\mathbb{F}_p)$, where $m=(n-1)q+1.$
\end{theorem}

\begin{proof}

Let $\alpha_{i,j} \in \mathbb{Q}$ be such that
\[
i < \alpha_{i,1} < \dots < \alpha_{i,q-1} < i+1, \qquad i=1,\dots,n-1,
\]
and let UT$_m(\mathbb{F}_p)$ be generated by
\[
t_{i, \alpha_{i,1}}', t_{\alpha_{i,1}, \alpha_{i,2}}', \dots, t_{\alpha_{i,q-1}, i+1}',
\qquad i=1,\dots,n-1.
\]
Write $a=(a_{i,j})$.
Using induction on $n$ we will prove that solution of~(\ref{eq:power_p_eq_over_UT_n}) with respect to embedding~(\ref{eq:UT_embedding_FR}) has the form

\begin{equation}
\label{eq:solution_form}
x = x_{n-1} x_{n-2} \dots x_2 x_1,
\end{equation}
where
\begin{align*}
    x_1 &= t_{\alpha_{1,q-1}, 2}' \dots t_{\alpha_{1,1},\alpha_{1,2}}' t_{1,\alpha_{1,1}}'(a_{1,2}), \\
    x_2 &= t_{\alpha_{2,q-1}, 3}' \dots t_{\alpha_{2,1},\alpha_{2,2}}' t_{2,\alpha_{2,1}}'(a_{2,3}) t_{1,\alpha_{2,1}}'(a_{1,3}), \\
    & \dots \\
    x_{n-1} &= t_{\alpha_{n-1,q-1}, n}' \dots t_{\alpha_{n-1,1},\alpha_{n-1,2}}' t_{n-1,\alpha_{n-1,1}}'(a_{n-1,n}) \cdot \\
    & \qquad \cdot t_{n-2,\alpha_{n-1,1}}'(a_{n-2,n}) \dots t_{1,\alpha_{n-1,1}}'(a_{1,n}).
\end{align*}

{\bf Induction base.} If $n=2$ then we obtain the equation
\begin{equation}
\label{eq:first_iter}
x^q = t_{1,2}(a_{1,2}).
\end{equation}
Make an insertion of indices $\alpha_{1,j} \in \mathbb{Q}$ such that
\[
1 < \alpha_{1,1} < \alpha_{1,2} < \dots < \alpha_{1,q-1} < 2.
\]
Consider the group UT$_{q+1}(\mathbb{F}_p)$, generated by transvections
\[
t_{1,\alpha_{1,1}}', t_{\alpha_{1,1}, \alpha_{1,2}}', \dots, t_{\alpha_{1,q-1}, 2}',
\]
and define the mapping $\phi_1:$ UT$_2(\mathbb{F}_p) \to $ UT$_{q+1}(\mathbb{F}_p)$ by
$
\phi_1(t_{1,2}) = t_{1,2}'.
$
In UT$_{q+1}(\mathbb{F}_p)$ take the element
\[
x_1 = t_{\alpha_{1,q-1}, 2}' \dots t_{\alpha_{1,1}, \alpha_{1,2}}' t_{1, \alpha_{1,1}}'(a_{1,2}),
\]
then $x_1^q = t_{1,2}'(a_{1,2})$.
Hence $x_1$ is a solution of~(\ref{eq:first_iter}).

{\bf Induction step.}
Suppose that the equation
\begin{equation}
\label{eq:n_iter}
x^q = a
\end{equation}
over UT$_{n+1}(\mathbb{F}_p)$ is given.
Consider this equation over the factor UT$_n(\mathbb{F}_p)=$ UT$_{n+1}(\mathbb{F}_p)/\mathrm{LC}_{n+1}$.
By induction the latter is solvable with respect to embedding $\phi$ of the form~(\ref{eq:UT_embedding_FR}) and its solution $x$ has the form~(\ref{eq:solution_form}).

Make an insertion of indices $\alpha_{n,j} \in \mathbb{Q}$ such that
\[
n < \alpha_{n,1} < \alpha_{n,2} < \dots < \alpha_{n,q-1} < n+1.
\]
Consider the embedding $\phi':$ UT$_{n+1}(\mathbb{F}_p) \to$ UT$_{nq+1}(\mathbb{F}_p)$ defined by
\begin{equation}
\label{eq:n_embedding}
\begin{aligned}
\phi'(t_{k,k+1}) &= \phi(t_{k,k+1}), \quad k=1,\dots,n-1, \\ 
\phi'(t_{n,n+1}) &= t_{n,n+1}' \underbrace{ t_{\alpha_{n-1,1},\alpha_{n,1}}' \dots t_{\alpha_{n-1,q-1},\alpha_{n,q-1}}'}_{\Delta_n}.
\end{aligned}
\end{equation}
Compute the image of $a$ under $\phi'$.
It is clear that
$$
a = \prod\limits_{i=1}^n t_{i,n+1}(a_{i,n+1}) \bar{a},
$$
where $\bar{a} \in$ B$_{n+1}$ (in notations of~(\ref{eq:subgroups_of_UT})).
Hence
\[
\phi'(a) = \phi' \left( \prod_{i=1}^n t_{i,n+1}(a_{i,n+1}) \right) \phi(\bar{a}).
\]
Observe that $\phi(\bar{a}) = x^q$.
Computing for $i=n-1,\dots,1$ the values $\phi'(t_{i,n+1}) = \phi'([t_{i,i+1},t_{i+1,n+1}])$ we obtain
\begin{align*}
\phi'(t_{1,n+1})
    &= t_{1,n+1}' = t_{1,n+1}' \Delta_1, \\
\phi'(t_{i,n+1}) 
    &= t_{i,n+1}' \underbrace{ t_{\alpha_{i-1,1},\alpha_{n,1}}' t_{\alpha_{i-1,2},\alpha_{n,2}}' \dots t_{\alpha_{i-1,q-1},\alpha_{n,q-1}}'}_{\Delta_i}, \quad i=2,\dots,n-1.
\end{align*}
Observe that $\Delta_i$ and $t_{i,n+1}'$ commute and also $\Delta_i$ and $\Delta_j$ commute.
Denote $\Delta_i(\gamma) = \Delta_i^\gamma$ then
$
\phi'(t_{i,n+1}(\gamma)) = t_{i,n+1}'(\gamma) \Delta_i(\gamma)
$
and
\[
\phi' \left( \prod_{i=1}^n t_{i,n+1}(a_{i,n+1}) \right)
    = \left( \prod_{i=1}^n t_{i,n+1}'(a_{i,n+1}) \right) \Delta_2(a_{2,n+1}) \dots \Delta_n(a_{n,n+1}).
\]
Take
\begin{align*}
x_n
    &= t_{\alpha_{n,q-1},n+1}' \dots t_{n,\alpha_{n,1}}'(a_{n,n+1})  t_{n-1,\alpha_{n,1}}'(a_{n-1,n+1}) \dots t_{1,\alpha_{n,1}}'(a_{1,n+1}) \\
    &= e + a_{1,n+1} e_{1, \alpha_{n,1}} + \dots + a_{n,n+1} e_{n,\alpha_{n,1}} + e_{\alpha_{n,1},\alpha_{n,2}} + \dots + e_{\alpha_{n,q-1},n+1} \\
    &= e + y,
\end{align*}
then
$
x_n^q = \prod\limits_{i=1}^n t_{i,n+1}'(a_{i,n+1})
$
and
\[
\phi'(a) = x_n^q \Delta_2(a_{2,n+1}) \dots \Delta_n(a_{n,n+1}) x^q.
\]
We will show that $x' = x_n x$ is a solution of~(\ref{eq:n_iter}) with respect to embedding~(\ref{eq:n_embedding}).
Write $x = e + z$.
Since $yz = 0$, we get
\begin{align*}
(x_n x)^q &= (e + y + z)^q \\
    &= e + y^q + z^q + z^{q-1}y + z^{q-2}y^2 + \dots + zy^{q-1} \\
    &= (e + y^q) (e + z^{q-1}y + z^{q-2}y^2 + \dots + zy^{q-1}) (e + z^q) \\
    &= x_n^q (e + z^{q-1}y + z^{q-2}y^2 + \dots + zy^{q-1}) x^q.
\end{align*}
Observe that
\[
z^{q-k}y^k = a_{2,n+1} e_{\alpha_{1,k},\alpha_{n,k}} +
a_{3,n+1} e_{\alpha_{2,k},\alpha_{n,k}} + \dots +
a_{n,n+1} e_{\alpha_{n-1,k},\alpha_{n,k}},
\]
hence
\[
e + z^{q-1}y + z^{q-2}y^2 + \dots + zy^{q-1} =
\Delta_2(a_{2,n+1}) \dots \Delta_n(a_{n,n+1}).
\]
Finally we obtain
$
(x_n x)^q = \phi'(a),
$
so $x_n x$ is a solution of~(\ref{eq:n_iter}) with respect to embedding~(\ref{eq:n_embedding}).

In a similar way one can prove that solution of~(\ref{eq:power_p_eq_over_UT_n}) with respect to embedding~(\ref{eq:UT_embedding_LC}) has the form
\[
x = x_1 x_2 \dots x_{n-1},
\]
where
\begin{align*}
    x_1 &= t_{\alpha_{n-1,q-1}, n}'(a_{n-1,n}) t_{\alpha_{n-1,q-2},\alpha_{n-1,q-1}}' \dots t_{n-1,\alpha_{n-1,1}}', \\
    x_2 &= t_{\alpha_{n-2,q-1}, n}'(a_{n-2,n}) t_{\alpha_{n-2,q-1}, n-1}'(a_{n-2,n-1}) t_{\alpha_{n-2,q-2},\alpha_{n-2,q-1}}' \dots t_{n-2,\alpha_{n-2,1}}', \\
    & \dots \\
    x_{n-1} &= t_{\alpha_{1,q-1}, n}'(a_{1,n}) \dots t_{\alpha_{1,q-1},2}'(a_{1,2}) t_{\alpha_{1,q-2},\alpha_{1,q-1}}' \dots t_{1,\alpha_{1,1}}'.
\end{align*}

\end{proof}

Observe that the theorem above performs {\it simultaneous} adjunction of $q$-th roots to UT$_n(\mathbb{F}_p)$, i.e., any element of UT$_n(\mathbb{F}_p)$ has a $q$-th root in UT$_m(\mathbb{F}_p)$.

\begin{example}
Let $n=p=3$ and
$
a = \left( \begin{smallmatrix}
1 & a_{12} & a_{13} \\
0 & 1      & a_{23} \\
0 & 0      & 1
\end{smallmatrix} \right) \in \mathrm{UT}_3(\mathbb{F}_3).
$
According to Theorem~\ref{th:roots_in_UT_n}, solution of $x^3=a$ over UT$_3(\mathbb{F}_3)$ 
with respect to embedding~(\ref{eq:UT_embedding_FR}) has the form
\[
x = \left( \begin{smallmatrix}
1 & a_{12} & 0 & 0 & a_{13} & 0 & 0 \\
0 & 1 & 1 & 0 & 0 & 0 & 0 \\
0 & 0 & 1 & 1 & 0 & 0 & 0 \\
0 & 0 & 0 & 1 & a_{23} & 0 & 0 \\
0 & 0 & 0 & 0 & 1 & 1 & 0 \\
0 & 0 & 0 & 0 & 0 & 1 & 1 \\
0 & 0 & 0 & 0 & 0 & 0 & 1
\end{smallmatrix} \right) \in \mathrm{UT}_7(\mathbb{F}_3).
\]
With respect to embedding~(\ref{eq:UT_embedding_LC}) solution has the form
\[
x = \left( \begin{smallmatrix}
1 & 1 & 0 & 0 & 0 & 0 & 0 \\
0 & 1 & 1 & 0 & 0 & 0 & 0 \\
0 & 0 & 1 & a_{12} & 0 & 0 & a_{13} \\
0 & 0 & 0 & 1 & 1 & 0 & 0 \\
0 & 0 & 0 & 0 & 1 & 1 & 0 \\
0 & 0 & 0 & 0 & 0 & 1 & a_{23} \\
0 & 0 & 0 & 0 & 0 & 0 & 1
\end{smallmatrix} \right) \in \mathrm{UT}_7(\mathbb{F}_3).
\]
\end{example}

\section{Embeddings of wreath products}
\label{sec:wp}

Let $t_{i,j}$ denote a transvection in the group UT$_n(\mathbb{F}_p)$, where $n \geq 2$, $q=p^s$, $s \in \mathbb{Z}^+$, and let $\alpha_{i,j} \in \mathbb{Q}$ be such that
\[
i < \alpha_{i,1} < \dots < \alpha_{i,q-1} < i+1, \qquad i=1,\dots,n-1.
\]
Let UT$_m(\mathbb{F}_p)$, where $m=(n-1)q+1$, be generated by
\[
t_{i, \alpha_{i,1}}', t_{\alpha_{i,1}, \alpha_{i,2}}', \dots, t_{\alpha_{i,q-1}, i+1}', \qquad i=1,\dots,n-1.
\]

\begin{lemma}
\label{lm:technical_1}
The mapping $\theta:$ UT$_n(\mathbb{F}_p) \to$ UT$_m(\mathbb{F}_p)$, defined by
\[
\theta (t_{i,i+1}) = t_{i,\alpha_{i,1}}' t_{i,\alpha_{i,2}}'^{-1} t_{i,\alpha_{i,3}}' \dots t_{i,\alpha_{i,q-1}}'^{-1} t_{i,i+1}', \qquad i=1,\dots,n-1,
\]
is an embedding.
\end{lemma}

\begin{proof}
From the following identity
\[
[x,yz]=[x,z][x,y][x,y,z]
\]
we obtain
\[
\theta(t_{i,j}) = t_{i,\alpha_{j-1,1}}' t_{i,\alpha_{j-1,2}}'^{-1} t_{i,\alpha_{j-1,3}}' \dots t_{i,\alpha_{j-1,q-1}}'^{-1} t_{i,j}'.
\]
It is easy to check that relations~(\ref{eq:UT_relations}) hold for $\theta(t_{i,j})$ and $\theta(g) \neq 1$ for $g \neq 1$.

\end{proof}

\begin{lemma}
\label{lm:technical_2}
Let $p$ be a prime, $q=p^s$, $s \in \mathbb{Z}^+$,
\[
A = \left( \begin{smallmatrix}
1 & 1 & 0 & \dots & 0 \\
 & 1 & 1 & \dots & 0 \\
\vdots && \ddots & \ddots & \vdots \\
 &&& 1 & 1\\
0 && \dots && 1
\end{smallmatrix} \right), \;
B = \left( \begin{smallmatrix}
0 & 0 & 0 & \dots & 0 & 0 \\
0 & 0 & 0 & \dots & 0 & 0 \\
\vdots & \vdots & \vdots && \vdots & \vdots \\
0 & 0 & 0 & \dots & 0 & 0 \\
1 & -1 & 1 & \dots & -1 & 1
\end{smallmatrix}\right) \in M_{q \times q}(\mathbb{F}_p)
\]
and $M_i = A^{-i} B A^i$, for $i=0,\dots,q-1$ ($M_0=B$).
Then the following holds:
\begin{enumerate}
\item[1)] $(1,-1,1,\dots,-1,1)M_i = (0,\dots,0)$ for $i=1,\dots,q-1$,
\item[2)] $\sum_{i=0}^{q-1} (1,-1,1,\dots,-1,1) A^i = (0,\dots,0,1)$,
\item[3)] $\sum_{i=0}^{q-1} M_i = E$.
\end{enumerate}
\end{lemma}

\begin{proof}
In the matrix $M_i$ each column $j$ $(j=2,\dots,q)$ is a multiple of the first one.
Indeed, for $M_0=B$ this statement holds and multiplications by $A^{-1}$ on the left and by $A$ on the right preserve this property.
Multiplication by $A$ on the right doesn't change the first column.
Hence, to prove the first statement it is enough to show that (notice that $A^{-i}=A^{q-i}$)
\[
(1,-1,1,\dots,-1,1) A^i \left(\begin{smallmatrix}
0 \\
\vdots \\
0 \\
1
\end{smallmatrix}\right) = 0, \quad for \ \ i=1,\dots,q-1.
\]
If $\mathbf{x}^T=(x_q,\dots,x_1)$ then $A\mathbf{x}=(x_q + x_{q-1}, \dots, x_2+x_1,x_1)$.
Further
\[
(1,-1,1,\dots,-1,1) \left(\begin{smallmatrix}
x_q + x_{q-1} \\
\vdots \\
x_2 + x_1 \\
x_1
\end{smallmatrix}\right) = x_q = 0,
\]
since in matrices
\[
\left(\begin{smallmatrix}
y_q \\
\vdots \\
y_2 \\
y_1
\end{smallmatrix}\right) = A^i \left(\begin{smallmatrix}
0 \\
\vdots \\
0 \\
1
\end{smallmatrix}\right), \quad i=1,\dots,q-1,
\]
the value of $y_q$ is equal to $0$.
This proves the first statement.

The second statement follows from the third one, since multiplication of $B$ by $A^{-1}$ on the left doesn't change the last row of $B$.

Denote $M = \sum_{i=0}^{q-1} M_i$.
Observe that all $M_i$ (and respectively $M$) are lower triangular matrices and $A^{-1}MA=M$.
Then from the system of linear equations $XA=AX$ on the unknown lower triangular matrix $X$ it follows that $X = \lambda E$.
Since $M_0 M = M_0$ we have $\lambda=1$ and $M=E$.
\end{proof}

\begin{theorem}
\label{th:main}
The wreath product UT$_n(\mathbb{F}_p) \wr C_q$ ($n \geq 2$) of unitriangular group UT$_n(\mathbb{F}_p)$ with the cyclic group of order $q=p^s$, $s \in \mathbb{Z}^+$, embeds in UT$_m(\mathbb{F}_p)$, where $m=(n-1)q+1$.
\end{theorem}

\begin{proof}
Let $t_{i,j}$ denote a transvection in the group UT$_n(\mathbb{F}_p)$.
Let $\alpha_{i,j} \in \mathbb{Q}$ be such that
\[
i < \alpha_{i,1} < \dots < \alpha_{i,q-1} < i+1, \qquad i=1,\dots,n-1,
\]
and let UT$_m(\mathbb{F}_p)$, where $m=(n-1)q+1$, be generated by
\[
t_{i, \alpha_{i,1}}', t_{\alpha_{i,1}, \alpha_{i,2}}', \dots, t_{\alpha_{i,q-1}, i+1}', \qquad i=1,\dots,n-1.
\]
By $a$ we denote the generator of $C_q$.
In UT$_m(\mathbb{F}_p)$ we will construct an element $c$ of order $q$ and subgroups $G_1,\dots,G_q$ such that the following conditions hold:
\begin{enumerate}
\item[1)] $G_{i+1} = c^{-1} G_i c$, for $i=1,\dots,q-1$;
\item[2)] $G_i \simeq$ UT$_n(\mathbb{F}_p)$, $\phi_i:$ UT$_n(\mathbb{F}_p) \to G_i$ is a corresponding isomorphism and $\phi_{i+1}(t_{j,j+1}) = c^{-1} \phi_i(t_{j,j+1}) c$, for $j=1,\dots,n-1$;
\item[3)] $G_i$ and $G_j$ are commuting element-wise for $i \neq j$;
\item[4)] $G_i \cap G_j = \{1\}$ for $i \neq j$.
\end{enumerate}
Then the mapping $\tau:$ UT$_n(\mathbb{F}_p) \wr C_q(a) \to$ UT$_m(\mathbb{F}_p)$, defined by
\begin{equation}\label{eq:wp_embedding}
\tau: a^k (h_1,h_2,\dots, h_q) \mapsto c^k \phi_1(h_1) \phi_2(h_2) \dots \phi_q(h_q),
\end{equation}
is an embedding.
Observe that according to Lemma~\ref{lm:wp_class} this value of $m$ is the minimal possible.

Denote $g_{i,j} = \phi_i(t_{j,j+1})$, for $i=1,\dots,q$, $j=1,\dots,n-1$.
To prove 3) we will show that $[g_{k,i},g_{l,j}]=1$ for $k \neq l$ and $i,j=1,\dots,n-1$.
Since $g_{k+1,i} = c^{-1} g_{k,i} c$ for $k=1,\dots,q-1$ and
\[
[g_{k,i},g_{l,j}]=1 \Longleftrightarrow [g_{1,i},c^{-(l-k)}g_{1,j}c^{l-k}]=1,
\]
it is enough to prove that $[g_{1,i},g_{l,j}]=1$ for $l=2,\dots,q$ and $i,j=1,\dots,n-1$.

From 3) it follows that to prove 4) it is enough to show that $\zeta(G_i) \cap \zeta(G_j) = \{1\}$ for $i \neq j$.

For $i=1,\dots,n-1$ denote
\begin{align}\label{eq:c}
c_i &= e + e_{\alpha_{i,1},\alpha_{i,2}} + e_{\alpha_{i,2},\alpha_{i,3}} + \dots + e_{\alpha_{i,q-1},i+1} \\ \nonumber
    &= t_{\alpha_{i,q-1},i+1}' \dots t_{\alpha_{i,2},\alpha_{i,3}}' t_{\alpha_{i,1},\alpha_{i,2}}'
\end{align}
and $c = c_1 c_2 \dots c_{n-1}$.
Clearly $[c_i,c_j]=1$ and $c_i$ has order $q$, hence $c$ has order $q$.

Define the ordered sets
\begin{align*}
I_1 &= \{ 1 \}, \\
I_i &= \{ \alpha_{i-1,1}, \alpha_{i-1,2}, \dots, \alpha_{i-1,q-1}, i \}, \qquad i=2,\dots,n.
\end{align*}

\noindent
Write
\begin{equation}\label{eq:h}
h_k = \prod_{ \substack{i \in I_k, \\ j \in I_{k+1}}} t_{i,j}'(\gamma_{i,j}),
\end{equation}
where $\gamma_{i,j} \in \mathbb{F}_p$ and $k=1,\dots,n-1$.
Observe that in the product above all transvections commute.
With $h_k$ we associate $|I_k| \times |I_{k+1}|$ matrix $M(h_k) = (\gamma_{i,j})$ over the field $\mathbb{F}_p$, with rows and column indexed by $I_k$ and $I_{k+1}$ respectively.
And conversely, with any such matrix we associate an element of the form~(\ref{eq:h}).
Further we will reduce operations with elements of the form~(\ref{eq:h}) to operations with corresponding matrices.

Observe that
\begin{align*}
c^{-1} h_1 c &= c_1^{-1} h_1 c_1, \\
c^{-1} h_k c &= c_{k-1}^{-1} c_k^{-1} h_k c_k c_{k-1}, \quad k=2,\dots,n-1.
\end{align*}

\noindent
From~(\ref{eq:UT_relations}) we obtain
\begin{align*}
t_{i,j}'(\alpha)^{t_{j,k}'(\beta)} &= t_{i,k}'(\alpha\beta) t_{i,j}'(\alpha), \quad \alpha,\beta \in \mathbb{F}_p, \\
t_{j,k}'(\beta)^{t_{i,j}'(\alpha)} &= t_{i,k}'(-\alpha\beta) t_{j,k}'(\beta), \\
t_{i,j}'(\alpha)^{t_{k,l}'(\beta)} &= t_{i,j}'(\alpha), \quad j \neq k, \; i \neq l.
\end{align*}

\noindent
Thus $c^{-1} h_k c$ and $h_k$ ($k=1,\dots,n-1$) are elements of the form~(\ref{eq:h}).
For $k=1,\dots,n-1$ we have
$
M(c_k^{-1}h_kc_k) = M(h_k) A
$
, where
\[
A = \left(\begin{smallmatrix}
1 & 1 & 0 & \dots & 0 \\
 & 1 & 1 & \dots & 0 \\
\vdots && \ddots & \ddots & \vdots \\
 &&& 1 & 1\\
0 && \dots && 1
\end{smallmatrix}\right) \in \mathrm{UT}_q(\mathbb{F}_p).
\]

\noindent
For $k=2,\dots,n-1$ we have
$
M(c_{k-1}^{-1} h_k c_{k-1}) = A^{-1} M(h_k).
$
Combining all the above we obtain
\begin{align*}
M(c^{-1}h_1c) &= M(h_1) A, \\
M(c^{-1}h_kc) &= A^{-1} M(h_k) A, \qquad k=2,\dots,n-1.
\end{align*}

Take
\begin{equation}\label{eq:h_value}
h_k = t_{k,\alpha_{k,1}}' t_{k,\alpha_{k,2}}'^{-1} t_{k,\alpha_{k,3}}' \dots t_{k,\alpha_{k,q-1}}'^{-1} t_{k,k+1}',
\end{equation}
then $M(h_1) = (1,-1,1,\dots,-1,1)$ (for $p=2$ we treat it as $(1,1,\dots,1)$)
and
\[
M(h_k) = \left(\begin{smallmatrix}
    0 & 0 & 0 & \dots & 0 & 0 \\
    \vdots & \vdots & \vdots && \vdots & \vdots \\
    0 & 0 & 0 & \dots & 0 & 0 \\
    1 & -1 & 1 & \dots & -1 & 1
\end{smallmatrix}\right), \quad k=2,\dots,n-1.
\]

Take $g_{1,k}=h_k$, for $k=1,\dots,n-1$, then by~Lemma \ref{lm:technical_1} the subgroup $G_1 = \langle g_{1,1},\dots,g_{1,n-1} \rangle$ is isomorphic to UT$_n(\mathbb{F}_p)$.
Further take $g_{i+1,k} = c^{-1} g_{i,k} c$, for $i=1,\dots,q-1$ and $k=1,\dots,n-1$, and $G_i = \langle g_{i,1},\dots,g_{i,n-1} \rangle$.
This proves statements 1) and 2).

Since $c^{-1} h_k c$ and $h_k$ are elements of the form~(\ref{eq:h}) then $[g_{1,k},g_{l,k}]=1$, for $k=1,\dots,n-1$, $l=2,\dots,q$.
It is also clear that $[g_{1,k},g_{l,j}]=1$ for $|j-k|>1$.
So it remains to consider the case when $|j-k|=1$.
Let $g_{1,k} = h_k = e + \mathcal{A}$ and $g_{l,k+1} = c^{-(l-1)} h_{k+1} c^{l-1} = e + \mathcal{B}$.
We will show that $\mathcal{AB}=\mathcal{BA}$.
Clearly $\mathcal{BA}=0$ and
\[
\mathcal{AB}=0 \Longleftrightarrow M(h_k) M(c^{-(l-1)} h_{k+1} c^{l-1}) = 0.
\]
The latter follows from the first statement of Lemma~\ref{lm:technical_2}.
Thus statement 3) is proved.

Observe that $\zeta(G_1) = \langle z_1 \rangle$, where
\[
z_1 = t_{1,\alpha_{n-1,1}}' t_{1,\alpha_{n-1,2}}'^{-1} \dots t_{1,\alpha_{n-1,q-1}}'^{-1} t_{1,n}'.
\]
Denote $z_l = c^{-(l-1)} z_1 c^{l-1}$, for $l=2,\dots,q$.
Define $y_i$ by
\[
(y_1,\dots,y_q) = (1,-1,\dots,-1,1) A^{l-1},
\]
then
\[
z_l = t_{1,\alpha_{n-1,1}}'^{y_1} t_{1,\alpha_{n-1,2}}'^{y_2} \dots t_{1,\alpha_{n-1,q-1}}'^{y_{q-1}} t_{1,n}'^{y_q}.
\]
The centers $\zeta(G_l) = \langle z_l \rangle$ are disjoint.
This proves statement 4) and brings our proof to the end.
\end{proof}

Let $\rho:$ UT$_n(\mathbb{F}_p) \to$ UT$_n(\mathbb{F}_p) \wr C_q$ be the embedding of UT$_n(\mathbb{F}_p)$ into the diagonal subgroup of the base group, $\tau:$ UT$_n(\mathbb{F}_p) \wr C_q(c) \to$ UT$_m(\mathbb{F}_p)$ be embedding~(\ref{eq:wp_embedding}), constructed in Theorem~\ref{th:main}, and $\phi:$ UT$_n(\mathbb{F}_p) \to$ UT$_m(\mathbb{F}_p)$ be embedding~(\ref{eq:UT_embedding_FR}).

\begin{lemma}\label{lm:equiv}
$\tau \circ \rho \equiv \phi$.
\end{lemma}

\begin{proof}
Using notations of Theorem~\ref{th:main} we will compute the diagonal subgroup of the base group.
Denote
\[
f_k = \prod_{l=1}^{q} g_{l,k}, \qquad k=1,\dots,n-1,
\]
then $f_k$ is an element of the type~(\ref{eq:h}).
Observe that
\begin{align*}
M(f_1) &= \sum_{l=1}^{q} M(g_{l,1})
        = \sum_{l=0}^{q-1} M(h_1) A^l, \\
M(f_k) &= \sum_{l=1}^{q} M(g_{l,k})
        = \sum_{l=0}^{q-1} A^{-l} M(h_k) A^l, \quad k=2,\dots,n-1.
\end{align*}
From Lemma~\ref{lm:technical_2} (statements 2) and 3)) it follows that
\begin{align*}
M(f_1) &= (0,\dots,0,1), \\
M(f_k) &= E, \qquad k=2,\dots,n-1.
\end{align*}
Hence $\phi(t_{i,i+1}) = \tau(\rho(t_{i,i+1}))$, for $i=1,\dots,n-1$.

\end{proof}

\bigskip
{\bf\large  Open question.}
Does there exist a nontrivial variety ${\mathbb L}$ of groups,  distinct from the variety ${\mathbb G}$ of all groups and the variety ${\mathbb A}$ of all abelian groups, such that every group $G \in {\mathbb L}$ is isomorphic to a subgroup of a divisible group $\bar{G} \in {\mathbb L}$?

\end{document}